\documentclass{amsart}
\usepackage[utf8]{inputenc}

\usepackage{mathptmx}

\usepackage{graphics}
\usepackage{thmtools}
\usepackage{wasysym}
\usepackage[T1]{fontenc}    

\usepackage{amsthm}
\usepackage{amsbsy,amsmath,amssymb,amscd,amsfonts,float}
\usepackage[pagebackref=true]{hyperref}
\usepackage[nameinlink,capitalize,noabbrev]{cleveref}

\usepackage{graphicx,float,latexsym,color}
\usepackage[font={scriptsize,it}]{caption}
\usepackage{subcaption}
\usepackage[export]{adjustbox}

\usepackage{makecell}
\renewcommand{\arraystretch}{1.2}

\usepackage[dvipsnames]{xcolor}

\newtheorem{theorem}{Theorem}
\newtheorem*{theorem*}{Theorem}
\newtheorem{observation}{Observation}
\newtheorem{proposition}{Proposition}
\newtheorem{conjecture}{Conjecture}
\newtheorem{corollary}{Corollary}
\newtheorem{lemma}{Lemma}
\theoremstyle{remark}
\newtheorem{remark}{Remark}
\theoremstyle{definition}
\newtheorem{definition}{Definition}

\hypersetup{
    pdftoolbar=true,        
    pdfmenubar=true,        
    pdffitwindow=false,     
    pdfstartview={FitH},    
    colorlinks=true,       
    linkcolor=OliveGreen,          
    citecolor=blue,        
    filecolor=black,      
    urlcolor=red           
}

\usepackage{lineno}

\usepackage{comment}
\usepackage{microtype}
\usepackage{footnote}

\newcommand{\E}{\mathcal{E}}
\newcommand{\T}{\mathcal{T}}
\newcommand{\C}{\mathcal{C}}

\renewcommand{\T}{\mathbb{T}}
\newcommand{\R}{\mathbb{R}}

\newcommand{\ol}{\overline}
\renewcommand{\l}{\lambda}
\newcommand{\X}{\mathcal{X}}

\newcommand{\ab}{_{\alpha,\beta}}

\usepackage[compact]{titlesec}
\titleformat{name=\section}{}{\thetitle.}{1em}{\centering\scshape}
\titlespacing{\section}{0pc}{1.5ex plus .1ex minus .2ex}{0pc}

\crefname{conjecture}{Conjecture}{Conjectures}

\title[Poncelet Triangles: a Theory for Locus Ellipticity]{Poncelet Triangles:\\a Theory for Locus Ellipticity\vspace{-0.5em}}
\author[M. Helman]{Mark Helman}
\thanks{M. Helman, Rice University,
Houston, USA. \texttt{markhelman@hotmail.com}}
\author[D. Laurain]{Dominique Laurain}
\thanks{D. Laurain, Enseeiht,
Toulouse, France. \texttt{dominique.laurain31@orange.fr}}
\author[R. Garcia]{Ronaldo Garcia}
\thanks{R. Garcia, Federal Univ. of Goiás, Brazil. \texttt{ragarcia@ufg.br}}
\author[D. Reznik]{Dan Reznik}
\thanks{D. Reznik$^*$, Data Science Consulting Ltd., Rio de Janeiro, Brazil. \texttt{dreznik@gmail.com}}
\date{October, 2021}

\begin{document}

\maketitle

\vspace{-1cm}
\begin{abstract}
We present a theory which predicts if the locus of a triangle center 
over certain Poncelet triangle families is a conic or not. We consider families interscribed in (i) the confocal pair and (ii) an outer ellipse and an inner concentric circular caustic. Previously, determining if a locus was a conic was done on a case-by-case basis. In the confocal case, we also derive conditions under which a locus degenerates to a segment or a circle. We show the locus' turning number is either $\pm 3$, while predicting its monotonicity with respect to the motion of a vertex of the triangle family. 
\end{abstract}

\section{Introduction}
Paraphrasing \cite[Thm 2.14, p.22]{dragovic11}, Poncelet's closure theorem says that given two real conics $\C,\C'$, if an $N$-gon can be drawn with all vertices on $\C$ and with all sides tangent to $\C'$, then a 1d family of such $N$-gons exists, where any point on $\C$ can be a vertex of the family. 

\begin{figure}[H]
    \centering
    \includegraphics[trim=0 0 0 0,clip,width=.5\textwidth]{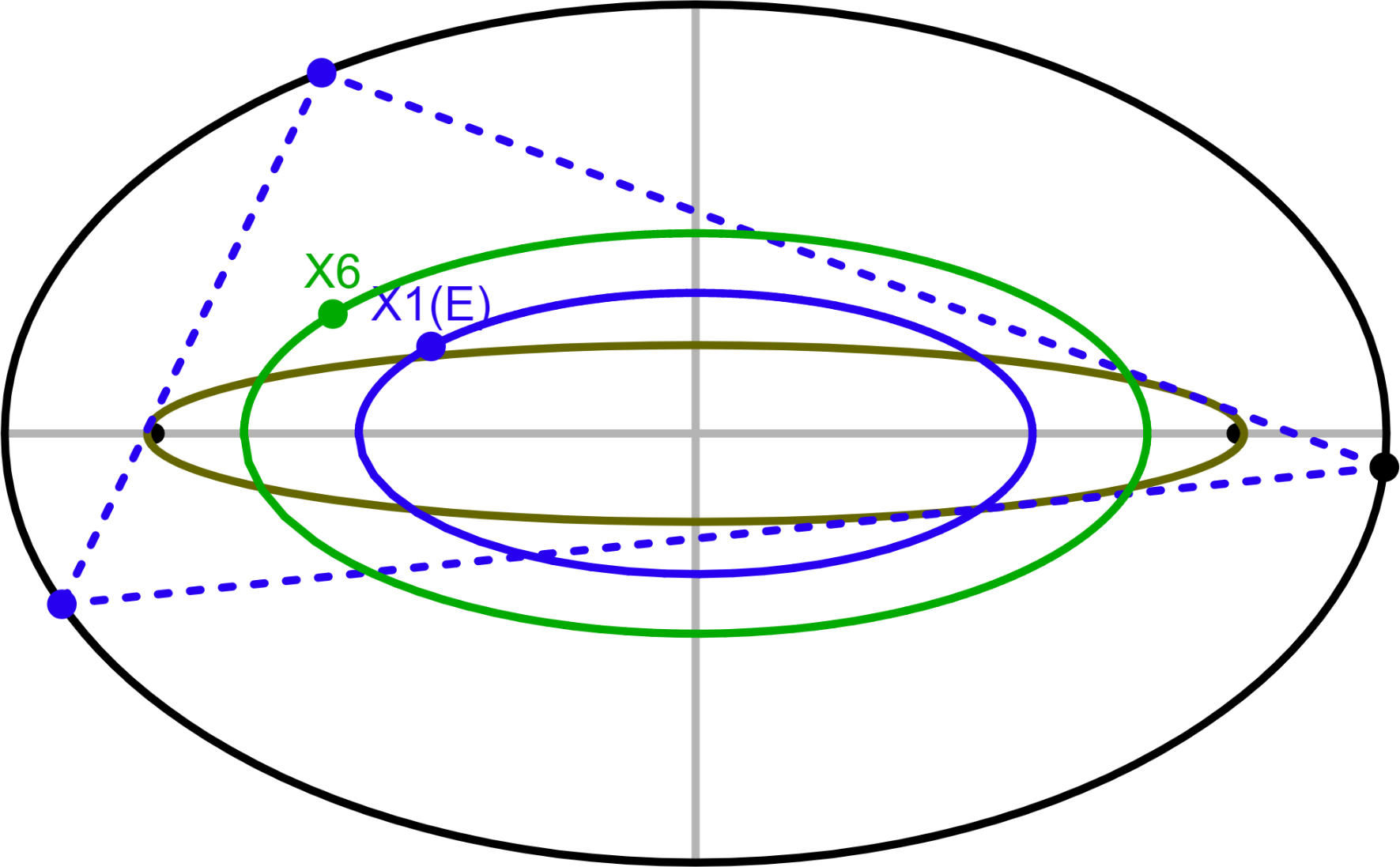}
    \caption{A Poncelet triangle (dashed blue) is shown embedded between two confocal ellipses. The locus of its incenter $X_1$ (resp. symmedian point $X_6$) is an ellipse (resp. quartic). \href{https://bit.ly/3hOad1q}{Live}}
    \label{fig:x1x6}
\end{figure}

Here we consider $N=3$ families, termed ``Poncelet triangles''. A first natural question is: what are curves swept by some of its notable points such as the incenter, barycenter, circumcenter, etc., or more generally by a vast array of its triangle centers, catalogued in \cite{etc}? A curious phenomenon, readily observed by simulation, is that some centers will sweep conics while others will not. 

As an example, \cref{fig:x1x6} shows Poncelet triangles embedded in a pair of confocal ellipses (also known as the ``elliptic billiard''). While in  \cite{garcia2019-incenter,olga14} it is shown that the incenter sweeps an ellipse, in \cite{garcia2020-new-properties} the locus of the symmedian point is a quartic. 

Though locus ellipticity can be painstakingly proved on a case-by-case basis, see for example \cite{corentin2021-circum,sergei2016-com}, we hereby describe a theory to explain the phenomenon over a wider range of cases. We employ a technique known as ``Blaschke's Parametrization'' which allows us to wield Poncelet triangles as roots of symmetric polynomials over the complex numbers.


\subsection*{Main results}

We show that if a triangle center $\X$ can be expressed as a fixed affine combination of barycenter, circumcenter, and $Y$, where $Y$ is stationary over a given family, then the locus of $\X$ will be an ellipse as well.

\begin{figure}
    \centering
    \includegraphics[width=\textwidth]{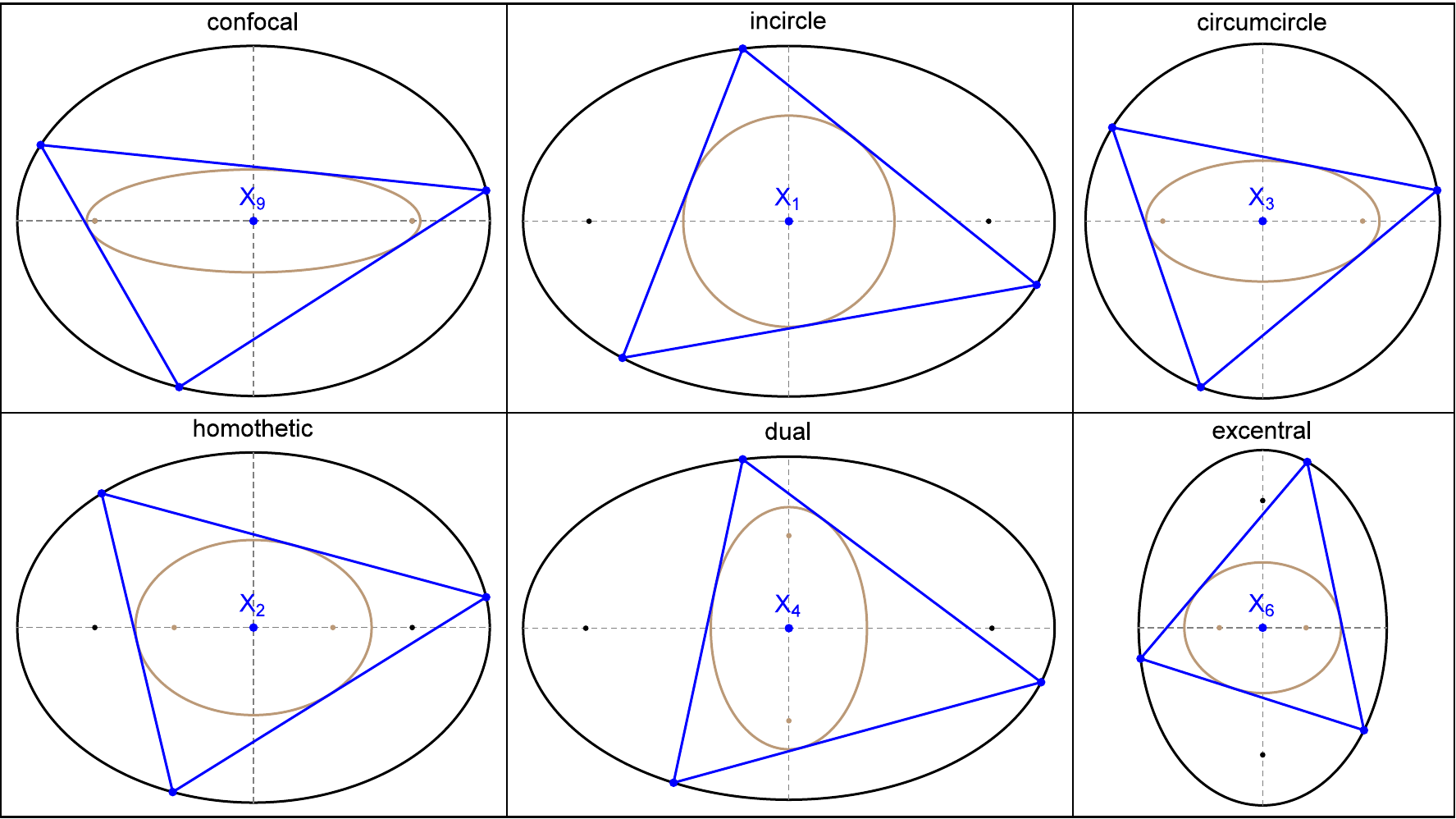}
    \caption{Six ``famous'' concentric, axis-parallel (CAP) families, and their stationary triangle centers at the center. \href{https://youtu.be/14TQ5WlZxUw}{Video}}
    \label{fig:six-caps}
\end{figure}

Referring to 6 Poncelet triangle families shown in
\cref{fig:six-caps}, let us henceforth refer to them as: confocal (elliptic billiard), with incircle, with circumcircle, homothetic, dual, and excentral, respectively.

Henceforth we shall refer to triangle centers as $X_k$, after Kimberling \cite{etc}. In \cite{garcia2020-family-ties,reznik2020-intelligencer} it was shown that the mittenpunkt $X_9$ (resp. symmedian point $X_6$) is stationary over the confocal (resp. excentral) family.

For the confocal case we (i) predict which triangle centers sweep ellipses, and (ii) derive conditions under which they degenerate to a circle or a segment. We then apply the same method for the incircle family, showing how to extend it to the excentral family. We still lack a strategy for locus-type prediction for the circumcircle, homothetic, and dual families, given that their stationary centers are already fixed linear combinations of barycenter $X_2$ and circumcenter $X_3$.

We also show that for a generic family, a locus' turning number is always $\pm 3$ and derive conditions for monotonicity with respect to the motion of family vertices.

To the best of our knowledge, the above results are new, e.g., they are not found in classical treatments such as \cite{dragovic11} nor in well-known surveys of Poncelet's theorem such as \cite{bos1987,dragovic2014,martini1995}.

\subsection*{Related Work}

In \cite{odehnal2011-poristic} loci of triangle centers are studied over the so-called ``poristic'' family (Poncelet triangles embedded between two circles\footnote{Also known as Chapple's porism.}). Conic, circular, and pointwise loci are identified. In \cite{sergei2016-com} the loci of vertex, perimeter, and area centroids are studied over a generic Poncelet family indicating that the first and last are always ellipses while in general the perimeter one is not a conic. The locus of the ``circumcenter-of-mass'' (a generalization of the circumcenter for N-gons), studied in \cite{sergei2014-circumcenter-of-mass}, is shown to be a conic over Poncelet N-gon families, see \cite{caliz2020-area-product}.
   
Over the confocal family, the elliptic locus of (i) the incenter was proved in \cite{corentin2021-circum,garcia2019-incenter,olga14}; (ii) of the barycenter in \cite{garcia2019-incenter,sergei2016-com}; and (iii) of the circumcenter in \cite{corentin2021-circum,garcia2019-incenter}. The elliptic locus of the Spieker center $X_{10}$ (which is the perimeter centroid of a triangle) was proved in \cite{garcia2019-incenter}.
Some properties and invariants of the confocal family are described in \cite{reznik2020-intelligencer}; $N=3$ subcases are proved in \cite{garcia2020-new-properties}. Some invariants have been proved for all $N{\geq}3$ in \cite{akopyan2020-invariants,bialy2020-invariants,caliz2020-area-product}. 

\subsection*{Article Structure} in \cref{sec:prelim} reviews preliminaries, and Blaschke's parametrization. \cref{sec:general} proves locus phenomena for triangles in a generic pair of ellipses. \cref{sec:confocal} analyzes loci of triangle centers in the confocal pair. In \cref{sec:incircle} this analysis is extended to the incircle family. We include links to animations in the caption of most figures, all of which are collected in \cref{tab:playlist}.

\cref{app:elliptic-loci} lists triangle centers whose loci are (numerically) ellipses (and/or circles) over other concentric, axis-parallel families.

\section{Preliminaries}
\label{sec:prelim}
Referring to \cref{fig:affine}, we regard Poncelet triangles in a generic ellipse pair as the image of a fixed affine transformation applied to the solutions of $B(z)=\lambda$ for each $\lambda$ in the unit circle of $\C$, where $B(\;\cdot\;)$ is a degree-3 Blaschke product \cite{daepp2019}. In \cite{helman2021-power-loci}, such a parametrization is used to derive the following fact (see \cref{fig:nonconcentric-xns}): 

\begin{theorem}
Over the family of Poncelet triangles interscribed in a generic nested pair of ellipses (non-concentric, non-axis-aligned),
if $\X\ab$ is a fixed linear combination of $X_2$ and $X_3$, i.e., $\X\ab=\alpha X_2+\beta X_3$ for some fixed $\alpha,\beta\in\mathbb{C}$, then its locus is an ellipse. 
\label{thm:ellipse-locus}
\end{theorem}

Recall the following observation reproduced from \cite{helman2021-power-loci}:

\begin{observation}
Amongst the 40k+ centers listed on \cite{etc}, about 4.9k triangle centers lie on the Euler line \cite{etc-central-lines}. Out of these, only 226 are always fixed affine combinations of $X_2$ and $X_3$. For $k<1000$, these amount to $X_k,k=${\small 2, 3, 4, 5, 20, 140, 376, 381, 382, 546, 547, 548, 549, 550, 631, 
632}.
\end{observation}

Consider the affine image of Poncelet triangles in a generic pair of ellipses such that the external ellipse is sent to the unit circle, \cref{fig:affine}. Let $f,g$ be the foci of the caustic thus obtained.

\begin{figure}
    \centering
    \includegraphics[width=\textwidth]{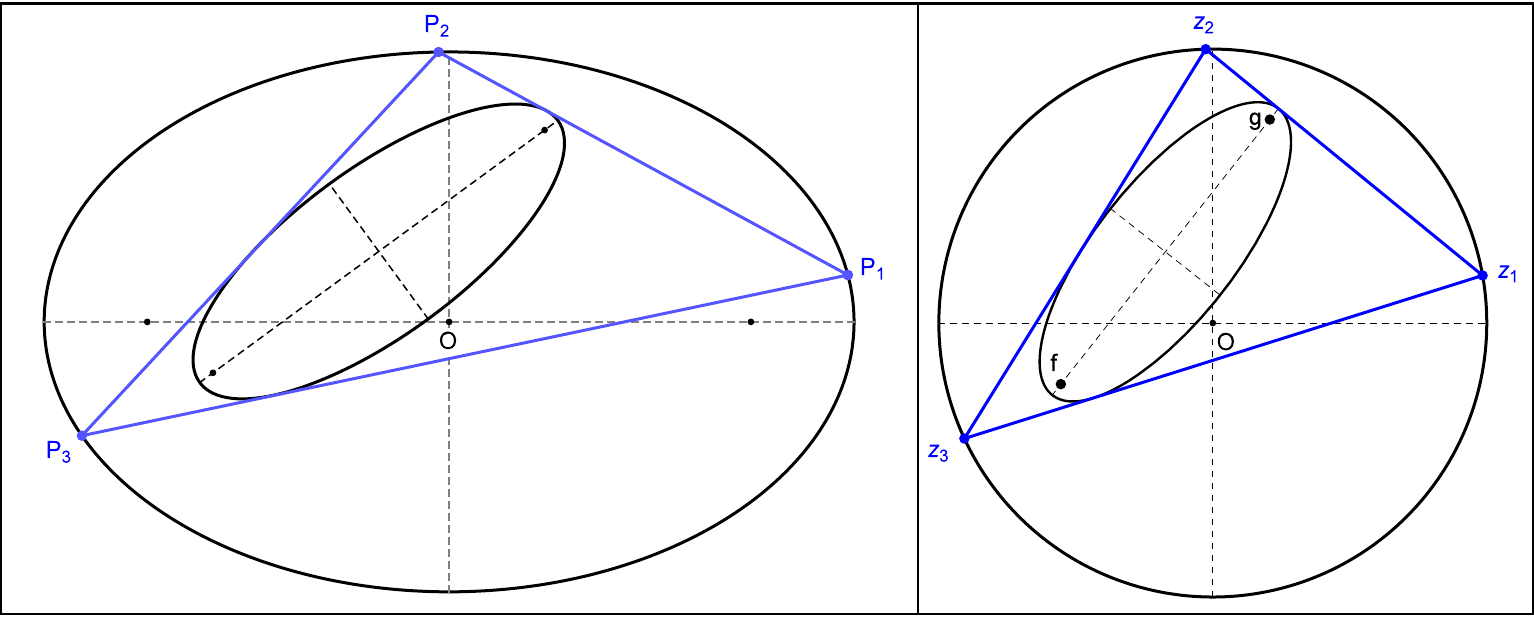}
    \caption{A generic ellipse pair which admits a Poncelet triangle family (left, vertices $P_i$) and its affine image such that the external conic is a circle (right, vertices are complex $z_i$, foci $f$ and $g$).  \href{https://youtu.be/6xSFBLWIkTM}{Video}}
    \label{fig:affine}
\end{figure}

Let $p=(a+b)/2$ and $q=(a-b)/2$. If, as in \cref{thm:ellipse-locus}, $\X\ab=\alpha X_2+\beta X_3$ for some fixed $\alpha,\beta\in\mathbb{C}$, then the latter can be parametrized as $ \X\ab= u \l+v\frac{1}{\l}+w$ where \cite[Eqn. 3]{helman2021-power-loci}:

\begin{align*}
    u:=&\frac{p \left(\ol{f} \ol{g} \left(\alpha  p^2-q^2 (\alpha +3 \beta )\right)+3 \beta  p q\right)}{3 (p-q) (p+q)}\\
    v:=&\frac{\beta  p q (q-f g p)}{(q-p) (p+q)}+\frac{1}{3} \alpha  f g q\\
    w:=&\frac{q \left(\ol{f}+\ol{g}\right) \left(p^2 (\alpha +3 \beta )-\alpha  q^2\right)+p (f+g) \left(\alpha  p^2-q^2 (\alpha +3 \beta )\right)}{3 (p-q) (p+q)}
\end{align*}


\begin{remark}
In \cite[Lemma 3.4, p. 28]{daepp2019} it is shown that (i) for each $\lambda$ on the complex unit circle there are 3 solutions for the equation $B(z)=\lambda$, where $B(z)$ is a degree-3 Blaschke product, and (ii) these 3 solutions move monotonically and in the same direction as $\lambda$. This means that as $\lambda$ sweeps the unit circle monotonically, Poncelet triangles sweep the outer ellipse monotonically and in the same direction as $\lambda$.
Moreover, as $\lambda$ sweeps the unit circle monotonically
each vertex of a Poncelet triangle sweeps the outer ellipse exactly once.
\label{rem:monotone-Blaschke}
\end{remark}

The following Lemma, proved in \cite{helman2021-power-loci}, was used to prove \cref{thm:ellipse-locus} and gives us a tool to explicitly find the semiaxis and rotation angle of said elliptic loci:

\begin{lemma}
If $u,v,w\in\mathbb{C}$ and $\l$ is a parameter that varies over the unit circle $\T\subset\mathbb{C}$, then the curve parametrized by
\[ F(\l)=u \l+v\frac{1}{\l}+w \]
is an ellipse centered at $w$, with semiaxis $|u|+|v|$ and $\big||u|-|v|\big|$, rotated with respect to the canonical axis of $\mathbb{C}$ by an angle of $(\arg u+\arg v)/2$.
\label{lem:ell-param}
\end{lemma}

\begin{definition}[Degenerate Locus] When the elliptic locus of a triangle center is either segment or a fixed point, i.e., either one or both of its axes have shrunk to zero, we will call it ``degenerate''.
\end{definition}

\section{Generic Ellipse Pair}
\label{sec:general}
In this section we consider Poncelet triangles in a generic pair of ellipses (non-concentric, and non-axis-parallel). We (i) specify a method to determine locus ellipticity if the family has a stationary center, (ii) we derive conditions for locus monotonicity, and (iii) compute the winding number of loci. 

\begin{corollary}
If a triangle center $X_k$ is stationary over a Poncelet triangle family, then the locus of any triangle center $\X$ which is a fixed linear combination of $X_2,X_3,X_k$ will be an ellipse. 
\label{cor:fixed-lin-comb}
\end{corollary}

\begin{proof}
The triangle center $\X=\alpha X_2+ \beta X_3+ \gamma X_k$ is the linear combination $\X\ab:=\alpha X_2+ \beta X_3$ under a fixed translation by $\gamma X_k$, because both $\gamma$ and $X_k$ are fixed over the family. The result then follows from \cref{thm:ellipse-locus}.
\end{proof}

\begin{proposition}
Let $\X$ be a fixed linear combination of $X_2$, $X_3$, and $X_k$, where $X_k$ is some stationary center over the family of Poncelet triangles. As the vertices of said triangles sweep the outer ellipse monotonically, the path of $\X$ in its elliptical locus is monotonic as well, except for when this locus is degenerate.
\end{proposition}

\begin{proof}
By \cref{thm:ellipse-locus}, the locus of $\X$ can be parametrized by $u \l+v\frac{1}{\l}+w$ for some $u,v,w\in\mathbb{C}$, where $\l$ sweeps the unit circle in $\mathbb{C}$ in the same direction as family vertices sweep the outer ellipse of the Poncelet pair (see \cref{rem:monotone-Blaschke}). We can thus parametrize $\X$ as $\X(t)=u e^{i t}+v e^{-i t}+w$. If either $u=0$ or $v=0$, it is clear from this parametrization that $\X$ sweeps its locus monotonically. Thus, we can now assume that $u\neq0$ and $v\neq0$.

Denoting $u=u_0+i u_1$ and $v=v_0+i v_1$ with $u_0,u_1,v_0,v_1\in\R$, we can directly compute
\[
    \left|\frac{d}{d t}\X(t)\right|^2= |u|^2+|v|^2 + 2 \sin (2 t) (u_2 v_1-u_1 v_2)-2 \cos (2 t) (u_1 v_1+u_2 v_2)
\]

Since $(u_2 v_1-u_1 v_2)^2+(u_1 v_1+u_2 v_2)^2=(u_1^2+u_2^2)(v_1^2+v_2^2)=|u|^2|v|^2$, there is some angle $\phi\in[0,2\pi)$ (the angle between the vectors $(u_1,u_2)$ and $(v_1,v_2)$) such that $u_1 v_1+u_2 v_2=|u| |v|\cos\phi$ and $u_2 v_1-u_1 v_2=|u| |v|\sin\phi$. Substituting this back in the previous equation, we derive
\begin{gather*}
     \left|\frac{d}{d t}\X(t)\right|^2=|u| |v|\left(\frac{|u|}{|v|}+\frac{|v|}{|u|}+2\sin(2t)\sin(\phi)-2\cos(2t)\cos(\phi)\right)=\\
     =|u| |v|\left(\frac{|u|}{|v|}+\frac{|v|}{|u|}-2\cos(2t+\phi)\right)\geq |u| |v|\left(\frac{|u|}{|v|}+\frac{|v|}{|u|}-2\right)
\end{gather*}

By AM-GM inequality, this last quantity is always strictly greater than 0 unless $|u|=|v|$, which only happens when the locus of $\X$ is degenerate (see \cref{lem:ell-param}). If $|u|\neq|v|$, we will have $\left|\frac{d}{d t}\X(t)\right|^2>0$, and hence the velocity vector never vanishes, meaning that the $\X$ sweeps its smooth locus monotonically, as desired.
\end{proof}

\begin{figure}
     \centering
     \includegraphics[trim=0 325 0 0,clip,width=.9\textwidth]{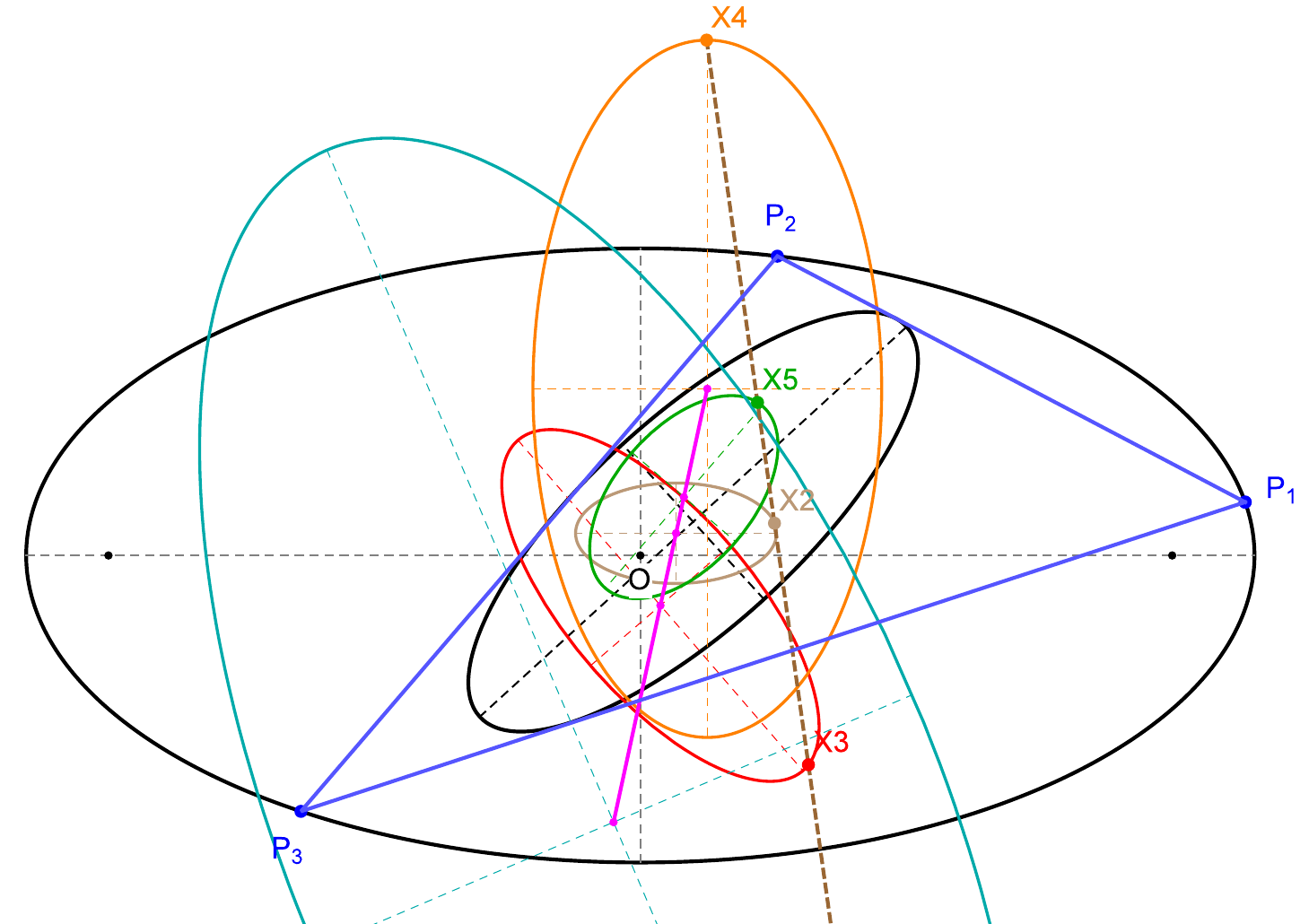}
     \caption{A Poncelet triangle is shown interscribed between two non-concentric, non-aligned ellipses (black). The loci of $X_k$, $k=2,3,4,5,20$ (the latter is only shown partially) are elliptic. Those of $X_2$ and $X_4$ are axis-aligned with the outer ellipse. Furthermore, the centers of all elliptic loci are collinear (magenta line). \href{https://youtu.be/p1medAei_As}{Video}}
     \label{fig:nonconcentric-xns}
 \end{figure}

\begin{proposition}
Let $\X$ be a fixed linear combination of $X_2$, $X_3$, and $X_k$, where $X_k$ is some stationary center over the family of Poncelet triangles. Over a full cycle of said triangles, the winding number of $\X$ over its elliptical locus is $\pm 3$, except for when this locus is degenerate.
\end{proposition}

\begin{proof}
By \cref{thm:ellipse-locus}, the locus of $\X$ can be parametrized by $u \l+v\frac{1}{\l}+w$ for some $u,v,w\in\mathbb{C}$. From \cref{rem:monotone-Blaschke}, one can see that the winding number of $\lambda$ associated to Poncelet triangles is $+3$ for each full counterclockwise cycle of said triangles over the outer Poncelet ellipse. Thus, it is sufficient to prove that the winding number of $\X$ over its elliptical locus is $\pm 1$ as $\lambda$ goes around the complex unit circle just once in the counterclockwise direction.

Since $w$ is the center of the elliptic locus of $\X$ (see \cref{lem:ell-param}), we compute the winding number of $\X$ around $w$. Parametrizing $\X$ as $\X(t)=u e^{i t}+v e^{-i t}+w$ where $\lambda=e^{i t}$, one can directly compute the winding number as \cite[Lemma 1, p. 114]{ahlfors1979-complex}:

\[
    \frac{1}{2\pi i}\oint_{\X}\frac{d\zeta}{\zeta-w}=\frac{1}{2\pi i}\int_0^{2\pi} \frac{\X'(t)}{\X(t)-w}d t={\mathop{\mathrm{sign}}}(|u|^2-|v|^2)
\]

By \cref{lem:ell-param}, the only way we can have $|u|=|v|$ is if the locus of $\X$ is degenerate. Thus, whenever this locus is not degenerate, the winding number of $\X$ around its locus as $\lambda$ sweeps the unit circle once is equal to $1$ if $|u|>|v|$ and $-1$ when $|u|<|v|$, as desired.
\end{proof}

\section{Confocal Pair}
\label{sec:confocal}
Consider a confocal ellipse which admits a family of Poncelet triangles. If $a,b$ denote the outer ellipse's semi-axes, the caustic semi-axes are given by \cite{garcia2019-incenter}:

\begin{equation*}
a_c=\frac{a\left(\delta-{b}^{2}\right)}{c^2},\;\;\;
b_c=\frac{b\left({a}^{2}-\delta\right)}{c^2}\cdot
\end{equation*}

\noindent where $c^2=a^2-b^2$ and $\delta=\sqrt{a^4-a^2 b^2+b^4}$.

The affine transformation  $x\rightarrow{x/a}$, $y\rightarrow{y/b}$ sends the confocal pair to a pair where the outer conic is the unit circle and the caustic $\E_c'$ is a concentric ellipse. The foci $f,g$ of the $\E_c'$ are given by:

\[ f=(-c',0),\;\;\;g=(c',0) \]

\noindent where $c'=(1/c)\sqrt{2 \delta -a^2-b^2}$.

We are now in a position to offer a compact alternate proof to the ellipticity of $X_1$ previously done in \cite{garcia2019-incenter,olga14}:

\begin{corollary}
In the confocal pair, the locus of $X_1$ is an ellipse.
\label{cor:confocal-x1}
\end{corollary}

\begin{proof}
For any triangle, $X_1$ can be expressed as the linear combination $X_1=\alpha X_2+\beta X_3+\gamma X_9$ of $X_2$, $X_3$ and $X_9$ with:

\[ \alpha =\frac{6}{\rho+2},\;\;\beta=\frac{2\rho}{\rho+2},\;\;\gamma=\frac{-\rho-4}{\rho+2} \]

\noindent where $\rho=r/R$, is the ratio of inradius to circumradius. Since in the confocal family $X_9$ is stationary and $\rho$ is invariant \cite{reznik2020-intelligencer}, the claim follows.
\end{proof}

Interestingly, and as originally stated in \cite{helman2021-power-loci}, ample experimental evidence suggests that:

\begin{conjecture}
The locus of $X_1$ is a non-degenerate conic iff the pair is confocal.
\label{conj:x1}
\end{conjecture}

We are now in a position to expand this last result to many other triangle centers in the confocal pair (elliptic billiard), as many of these are fixed linear combinations of $X_2$, $X_3$, and $X_9$. In order to do so, we compute several of such linear combinations in \cref{tab:abg}, most of which were derived from the ``combos'' in \cite{etc}. We use the existence of these linear combinations to prove the following very comprehensive result:

\setlength{\tabcolsep}{2pt}
\renewcommand{\arraystretch}{1.3}
\begin{table}
\begin{flushleft}
\begin{tabular}{|c|c|c|c|c|c|c|c|c|c|c|c|c|c|c|c|c|c|}
\hline
& $X_1$ & $X_2$ & $X_3$ & $X_4$ & $X_5$ & $X_7$ & $X_8$ & $X_9$ & $X_{10}$ & $X_{11}$ & $X_{12}$ & $X_{20}$ & $X_{21}$ & $X_{35}$ & $X_{36}$ & $X_{40}$  & $X_{46}$ 
\\ \hline
$\alpha$ & $1$ & $0$ & $0$ & $0$ & $0$ & $\frac{2\rho+4}{\rho+4}$ & $-2$ & $\frac{-\rho-2}{\rho+4}$ & $-\frac{1}{2}$ & $\frac{1}{1-2\rho}$ & $\frac{1}{1+2\rho}$ & $0$ & $0$ & $\frac{1}{2\rho+1}$ & $\frac{1}{1-2\rho}$ & $-1$
 & $\frac{1+\rho}{1-\rho}$
\\
\hline
$\beta$ & $0$ & $1$ & $0$ & $3$ & $\frac{3}{2}$ & $\frac{3\rho}{\rho+4}$ & $3$ & $\frac{6}{\rho+4}$ & $\frac{3}{2}$ & $\frac{-3\rho}{1-2\rho}$ & $\frac{3\rho}{1+2\rho}$ & $-3$ & $\frac{3}{2\rho+3}$ & $0$ & $0$ & $0$   & $0$ 
\\ \hline
$\gamma$ & $0$ & $0$ & $1$ & $-2$ & $-\frac{1}{2}$ & $\frac{-4\rho}{\rho+4}$ & $0$ & $\frac{2\rho}{\rho+4}$ & $0$ & $\frac{\rho}{1-2\rho}$ & $\frac{-\rho}{1+2\rho}$ & $4$ & $\frac{2\rho}{2\rho+3}$ & $\frac{2\rho}{2\rho+1}$ & $\frac{-2\rho}{1-2\rho}$ & $2$ 
 & $\frac{-2\rho}{1-\rho}$ 
\\
\hline 
\end{tabular}

\vspace{.5 em}
\begin{tabular}{|c|c|c|c|c|c|c|c|c|c|c|c|c|c|c|}
\hline
& $X_{55}$ & $X_{56}$ & $X_{57}$ & $X_{63}$ & $X_{65}$ & $X_{72}$ & $X_{78}$ & $X_{79}$ & $X_{80}$ & $X_{84}$ & $X_{90}$ & $X_{100}$ & $X_{104}$ & $X_{119}$ \\
\hline
$\alpha$ & $\frac{1}{1+\rho}$ & $\frac{1}{1-\rho}$ & $\frac{2+\rho}{2-\rho}$ & $\frac{-\rho-2}{\rho+1}$ & {\scriptsize $\rho+1$} & {\scriptsize $-\rho-2$} & $\frac{\rho+2}{\rho-1}$ & 1 & $\frac{2\rho+1}{1-2\rho}$ & $\frac{-\rho-2}{\rho}$ & $\frac{-(\rho+1)^2}{\rho^2+2\rho-1}$ & $\frac{2}{2\rho-1}$ & $\frac{-2}{2\rho-1}$ & $\frac{1}{2\rho-1}$ \\
\hline
$\beta$ & $0$ & $0$ & $0$ & $\frac{3}{\rho+1}$ & $0$ & $3$ & $\frac{-3}{\rho-1}$ & $\frac{6\rho}{2\rho+3}$ & $\frac{-6\rho}{1-2\rho}$ & $\frac{6}{\rho}$ & $\frac{6\rho}{\rho^2+2\rho-1}$ & $\frac{-3}{2\rho-1}$ & $\frac{3}{2\rho-1}$ & $\frac{3\rho-3}{2\rho-1}$ \\
\hline
$\gamma$ & $\frac{\rho}{1+\rho}$ & $\frac{-\rho}{1-\rho}$ & $\frac{-2\rho}{2-\rho}$ & $\frac{2\rho}{\rho+1}$ & {\small $-\rho$} & {\small $\rho$} & $0$ & $\frac{-6\rho}{2\rho+3}$ & $\frac{2\rho}{1-2\rho}$ & $\frac{2\rho-4}{\rho}$ & $\frac{2\rho(\rho-1)}{\rho^2+2\rho-1}$ & $\frac{2\rho}{2\rho-1}$ & $\frac{2\rho-2}{2\rho-1}$ & {$\frac{-\rho+1}{2\rho-1}$}\\
\hline
\end{tabular}

\vspace{.5 em}
\begin{tabular}{|c|c|c|c|c|c|c|c|c|c|}
\hline
& $X_{140}$ & $X_{142}$ & $X_{144}$ & $X_{145}$ & $X_{149}$ & $X_{153}$ & $X_{165}$ & $X_{191}$ & $X_{200}$ \\
\hline
$\alpha$ & $0$ & $\frac{\rho+2}{2\rho+8}$ & $\frac{-4\rho-8}{\rho+4}$ & $\frac{4}{7}$ & $\frac{-4}{6\rho-3}$ & $\frac{4}{6\rho-3}$ & $-\frac{1}{3}$ & $-1$ & $\frac{\rho+4}{\rho-2}$ \\
\hline 
$\beta$ & $\frac{3}{4}$ & $\frac{3\rho+6}{2\rho+8}$ & $\frac{12-3\rho}{\rho+4}$ & $\frac{3}{7}$ & $\frac{-6\rho+9}{6\rho-3}$ & $\frac{-6\rho-3}{6\rho-3}$ & $0$ & $\frac{6}{2\rho+3}$ & $\frac{-6}{\rho-2}$ \\
\hline
$\gamma$ & $\frac{1}{4}$ & $\frac{-2\rho}{2\rho+8}$ & $\frac{8\rho}{\rho+4}$ & $0$ & $\frac{12\rho-8}{6\rho-3}$ & $\frac{12\rho-4}{6\rho-3}$ & $\frac{4}{3}$ & $\frac{4\rho}{2\rho+3}$ & {$0$} \\
\hline
\end{tabular}
\caption{Triples $\alpha,\beta,\gamma$ used to express a given triangle center $X_k$ as the linear combinations $\alpha X_1+\beta X_2+\gamma X_3$. Note: $\rho=r/R$. Note also that though the loci of $X_{88}$, $X_{162}$, and $X_{190}$ are ellipses over the confocal family (in fact, they sweep the elliptic billiard), they are not included since they are not fixed linear combinations.}
\label{tab:abg}
\end{flushleft}
\end{table}


\begin{corollary}
In the confocal pair, from $X_1$ to $X_{200}$, the locus of $X_k$ is an ellipse for $k=$1,  2,  3,  4,  5,  7,  8,  10,  11,  12,  20,  21,  35,  36,  40,  46,  55,  56,  57,  63,  65,  72,  78,  79,  80,  88$^\dagger$, 84,  90, 100, 104, 119, 140, 142, 144, 145, 149, 153, 162$^\dagger$, 165, 190$^\dagger$, 191, 200.
\end{corollary}

\begin{proof}
As in \cref{cor:confocal-x1}, one can write $X_1$ as a fixed linear combination of $X_2$, $X_3$, and $X_9$, given that the ratio $\rho=r/R$ is constant in the confocal pair. Using the formulas on \cref{tab:abg}, all these triangle centers (except for $X_{88}$, $X_{162}$, and $X_{190}$) are fixed linear combinations of $X_1$, $X_2$, and $X_3$, and therefore they are fixed linear combinations of $X_2$, $X_3$, and $X_9$ as well. By \cref{cor:fixed-lin-comb}, given that $X_9$ is stationary over the confocal family, this implies the loci of all these triangle centers are ellipses.
\end{proof}

$^\dagger$Note: the loci of $X_{88}$, $X_{162}$, and $X_{190}$ are also ellipses because by definition they lie on the circumconic centered on $X_9$ \cite[X(9)]{etc}. Centers railed to said circumconic were called ``swans'' in \cite{reznik2020-ballet}, as over the confocal family they execute an intricate dance, which can be seen live \href{https://bit.ly/3yTGxqe}{here}.

Referring to \cref{fig:confocal-degenerate}:

\begin{figure}
    \centering
    \includegraphics[trim=0 0 0 60,clip,width=.9\textwidth]{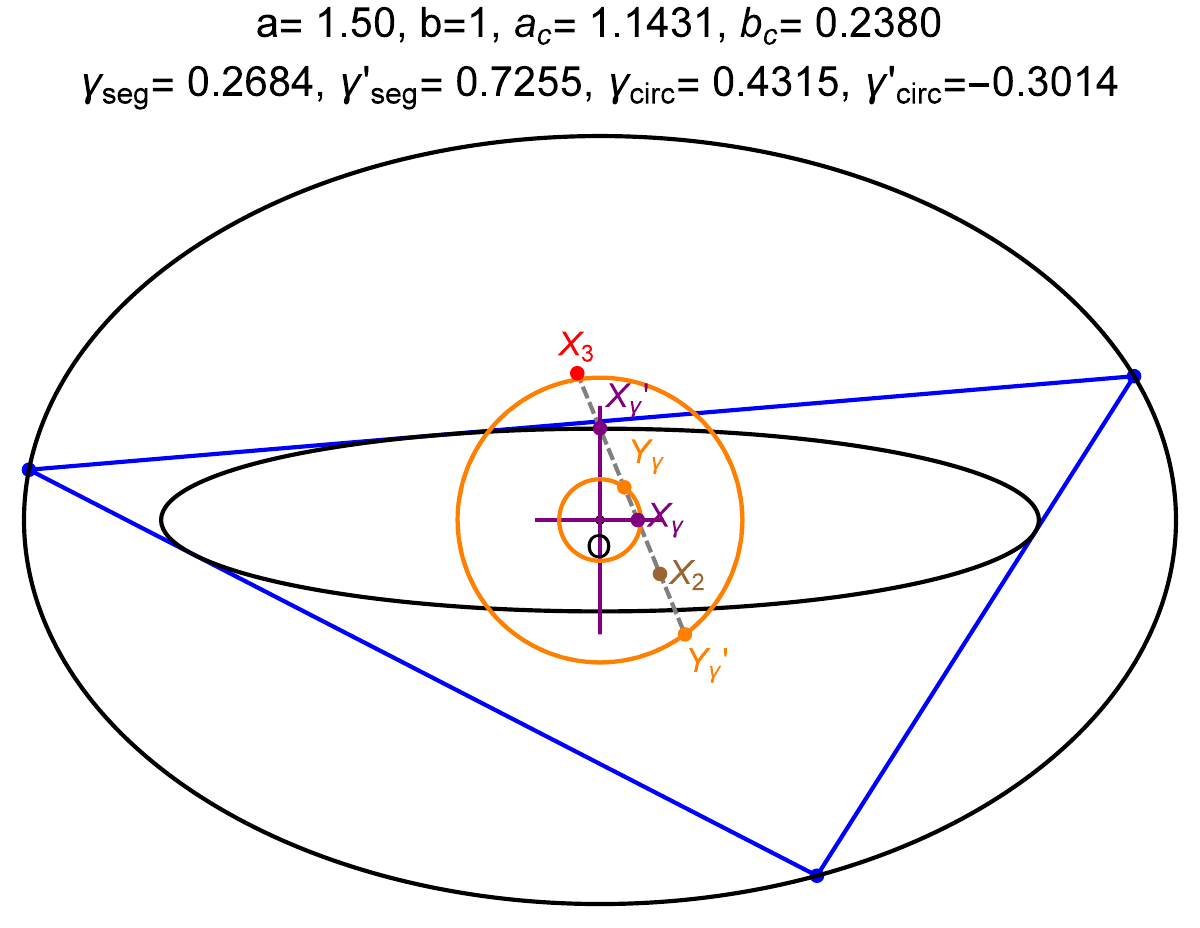}
    \caption{A Poncelet triangle (blue) in a pair of confocal ellipses (black) with $a/b=1.5$. Also shown are two degenerate (segment-like) loci (purple) obtained with $\gamma{\simeq}\{.27,.73\}$ and two circular loci (orange), obtained with $\gamma{\simeq}\{.43,-.3\}$. \href{https://youtu.be/haFTsq5UyK4}{Video}}
    \label{fig:confocal-degenerate}
\end{figure}

\begin{proposition}
In the confocal pair, the locus of $\X=\alpha X_2+\beta X_3$ for $\alpha,\beta\in\R$ is degenerate when:
\[\frac{\alpha}{\beta}=\frac{2 a^2-b^2+\delta }{2 b^2},\;\;\text{or}\;\;\frac{\alpha}{\beta}=\frac{2 b^2-a^2+\delta }{2 a^2} \]
\end{proposition}

\begin{proof}
By \cref{lem:ell-param}, this will happen when $|u|=|v|$ in \cref{thm:ellipse-locus}. In the confocal pair, when $\alpha,\beta\in\R$, both $u$ and $v$ are real numbers as well. Thus, the ratios $\alpha/\beta$ that yield degenerate loci can be computed directly by solving $u=\pm v$.
\end{proof}

\begin{observation}
These ratios $\alpha/\beta$ that make the locus of $\X$ be degenerate can also be expressed in terms of $\rho=r/R$ as: 
\[\frac{\alpha}{\beta}=\frac{3}{2}\left(\frac{ 1\pm\sqrt{1-2 \rho }}{ \rho +1\mp \sqrt{1-2 \rho }}\right) \]
\end{observation}

\begin{proposition}
In the confocal pair, the locus of $\X=\alpha X_2+\beta X_3$ for $\alpha,\beta\in\R$ is a circle when:

\[\left(\frac{\alpha}{\beta}\right)_{\pm}=\frac{\delta-3 a b\pm 2\left(a^2+b^2\right)}{2 a b} \]
\end{proposition}

\begin{proof}
By \cref{lem:ell-param}, this will happen when $|u|+|v|=\big||u|-|v|\big|$ with $u,v$ from \cref{thm:ellipse-locus}. In the confocal pair, when $\alpha,\beta\in\R$, both $u$ and $v$ are real numbers as well. Thus, this condition holds if and only if either $u=0$ or $v=0$. The ratios $\alpha/\beta$ that yield circular loci can then be computed directly.
\end{proof}

\begin{observation}
It follows that $\left({\alpha}/{\beta}\right)_+ +\left({\alpha}/{\beta}\right)_-=-3$.
\end{observation}

\section{Incircle Family and Beyond}
\label{sec:incircle}
Triangle centers whose loci are (numerically) ellipses (and/or circles) over other concentric, axis-parallel (CAP) families, appear in \cref{app:elliptic-loci}. Here we adapt the method in the previous section to  these families. 

Let the ``incircle family'', shown in \cref{fig:six-caps}, is a Poncelet triangle family interscribed in an external ellipse and a concentric internal circle. 

\begin{proposition}
In the incircle family, from $X_1$ to $X_{200}$, the locus of $X_k$ is an ellipse for $k=$2,  3,  4,  5,  7,  8,  9,  10,  11,  12,  20,  21,  35,  36,  40,  46,  55,  56,  57,  63,  65,  72,  78,  79,  80,  84,  90, 100, 104, 119, 140, 142, 144, 145, 149, 153, 165, 191, 200.
\end{proposition}

\begin{proof}
It \cite[Thm. 1]{garcia2020-family-ties} is is shown that the circumradius $R$ is constant over the incircle family. Since the inner ellipse is a circle with constant radius $r$, the ratio $\rho=r/R$ is constant over the incircle family. Using the formulas from \cref{tab:abg}, all these triangle centers are fixed linear combinations of $X_1$, $X_2$, and $X_3$. By \cref{cor:fixed-lin-comb}, given that $X_1$ is stationary over the incircle family, this implies the loci of all these triangle centers are ellipses.
\end{proof}

\noindent Note: $X_1$, a fixed point over the incircle family, can be regarded as a degenerate ellipse.

Referring to \cref{fig:six-caps}(bottom right), since in the excentral family $X_6$ is stationary, we could use a similar strategy to prove the ellipticity of certain centers. Nevertheless, we still lack a theory for this case.

The following is suggested by experimental results:

\begin{conjecture}
If a triangle center's barycentric coordinates are rational on the squares of a triangle's sidelengths, its locus will be an ellipse.
\label{conj:rational}
\end{conjecture}

The following are open questions:

\begin{itemize}
\item Why over the incircle family are the loci of certain centers circles?
\item How can we predict locus ellipticity in the homothetic, circumcircle, and dual families? Note $X_2$, $X_3$, and $X_4$ are stationary, so we don't have a third stationary center, independent of $X_2$ and $X_3$ which would allow us to apply our method.
\item Extend \cref{app:elliptic-loci} to well-known non-concentric pairs such as the Brocard porism \cite{reznik2020-similarityII}, the bicentric family (Chapple's porism), the MacBeath family (excentral triangles to  bicentrics family), etc.
\item Prove \cref{conj:x1,conj:rational}.
\end{itemize}

\subsection*{Videos}
Animations illustrating some phenomena herein are listed on Table~\ref{tab:playlist}.

{\small
\begin{table}
\begin{tabular}{|c|l|l|}
\hline
id & Title & \textbf{youtu.be/<.>}\\
\hline
01 & {Elliptic Loci of Triangles Centers in Generic Pair} &
\href{https://youtu.be/p1medAei_As}{\texttt{p1medAei\_As}}\\
02 & {Poncelet Triangles in Generic Pair + Affine Image w/ Circumcircle} &
\href{https://youtu.be/6xSFBLWIkTM}{\texttt{6xSFBLWIkTM}}\\
03 & {Triangle Centers with Circular and Segment-Like Loci} &
\href{https://youtu.be/haFTsq5UyK4}{\texttt{haFTsq5UyK4}}\\
04 & {Loci of Incenter and Excenters in a Generic Ellipse Pair} &
\href{https://youtu.be/z7qDgJEgPVY}{\texttt{z7qDgJEgPVY}}\\
\hline
05 & {Cayley-Poncelet Phenomena I: Basics} &
\href{https://youtu.be/virCpDtEvJU}{\texttt{virCpDtEvJU}}\\
06 & {Cayley-Poncelet Phenomena II: Intermediate} &
\href{https://youtu.be/4xsm\_hQU-dE}{\texttt{4xsm\_hQU-dE}}\\
\hline
\end{tabular}
\caption{Videos of some of the phenomena herein. The last column is clickable and provides the YouTube code.}
\label{tab:playlist}
\end{table}
}



\section*{Acknowledgements}
\noindent We would like to thank A. Akopyan for valuable insights. The first author is fellow of CNPq and coordinator of Project PRONEX/ CNPq/ FAPEG 2017 10 26 7000 508.

\appendix


\section{Triangle centers whose locus is an ellipse}
Referring to \cref{fig:six-caps}, here we report lists of Kimberling centers $X_k$ \cite{etc} whose locus under various classic Poncelet triangle families are ellipses (or circles).

{\small
\begin{itemize}
    \item Confocal pair (stationary $X_9$): Ellipses: 1,  2,  3,  4,  5,  7,  8,  10,  11,  12,  20,  21,  35,  36,  40,  46,  55,  56,  57,  63,  65,  72,  78,  79,  80,  84,  88,  90,  100,  104,  119,  140,  142,  144,  145,  149,  153,  162,  165,  190,  191,  200. Note: the first 29 in the list were proved in \cite{garcia2020-ellipses}. Note: the following centers lie on the $X_9$-centered circumellipse: 88, 100, 162, 190 \cite{etc}.
    \item Incircle (stationary $X_1$):  Ellipses: 2,  4,  7,  8,  9,  10,  20,  21,  63,  72,  78,  79,  84,  90,  100,  104,  140,  142,  144,  145,  149,  153,  191,  200. Circles: 3,  5,  11,  12,  35,  36,  40,  46,  55,  56,  57,  65,  80,  119,  165.
    \item Circumcircle (stationary $X_3$):
	Ellipses: 6, 49, 51,  52,  54,  64, 66,  67,  68,  69,  70,  113,  125,  141,  143,  146,  154,  155,   159,  161,  182,  184,  185,  193,  195. Circles: 2,  4,  5,  20,  22,  23,  24,  25,  26,  74,  98,  99,  100,  101,  102,  103,  104,  105,  106,  107,  108,  109,  110,  111,  112,  140,  156,  186.
    \item Homothetic (stationary $X_2$): Ellipses: 3,  4,  5,  6,  17, 20, 32,  39,  62,  69,  76,  83,  98,  99,  114,  115,  140,  141,  147,  148,  182, 183, 187,  190,  193,  194. Circles: 13,  14,  15,  16.
	\item Dual (stationary $X_4$) Ellipses: 2,  3,  5,  20,  64,  107,  122,  133,  140,  154.
    \item Excentral (stationary $X_6$): Ellipses: 2, 3, 4, 5, 20, 22, 23, 24, 25, 26, 49, 51, 52, 54, 64, 66, 67, 68, 69, 70, 74, 110, 113, 125, 140, 141, 143, 146, 154, 155, 156, 159, 161, 182, 184, 185, 186, 193, 195.
\end{itemize}
}

\label{app:elliptic-loci}


\bibliographystyle{maa}
\bibliography{999_refs,999_refs_rgk,999_refs_rgk_media}

\end{document}